\documentclass[12pt,a4paper,reqno]{amsart}
\usepackage[foot]{amsaddr}

\usepackage{amssymb}
\usepackage{amscd}
\usepackage{amsthm}
\usepackage[pdftex,pdfpagelabels]{hyperref}
\usepackage{enumerate}
\usepackage{graphicx}
\usepackage{mathtools}
\usepackage{siunitx}
\usepackage{tikz-cd}
\usepackage{bm}
\usepackage{mathtools}
\usepackage[tableposition=top]{caption}
\usepackage{booktabs,dcolumn}

\newcommand{\SL}{\mathrm{SL}}
\newcommand{\GL}{\mathrm{GL}}
\newcommand{\GF}{\mathrm{GF}}

\renewcommand{\mod}{\bmod}

\title[The Special Linear Group \& Block Unitriangular Matrices]{Representing the Special Linear Group with Block Unitriangular Matrices}
\author{John Urschel}
\address{\vspace{-.5mm}\newline Department of Mathematics, Massachusetts Institute of Technology, Cambridge, MA \newline Society of Fellows, Harvard University, Cambridge, MA}
\email{urschel@mit.edu}
 
\newtheorem{theorem}{Theorem}[section]

\newtheorem{lemma}[theorem]{Lemma}

\newtheorem{corollary}[theorem]{Corollary}
\begin{document}

\maketitle

\begin{abstract}
We prove that every element of the special linear group can be represented as the product of at most six block unitriangular matrices, and that there exist matrices for which six products are necessary, independent of indexing. We present an analogous result for the general linear group. These results serve as general statements regarding the representational power of alternating linear updates. The factorizations and lower bounds of this work immediately imply tight estimates on the expressive power of linear affine coupling blocks in machine learning.
\end{abstract}

\section*{Introduction}
Let $\mathbb{F}$ be an arbitrary field and  $\mathrm{M}_{m,n}(\mathbb{F})$ denote the set of $m \times n$ matrices with coefficients in $\mathbb{F}$. When $m =n$, we simply write $\mathrm{M}_{n}(\mathbb{F})$. Let $\mathrm{GL}_n(\mathbb{F})$ denote the group of $n \times n$ non-singular matrices with coefficients in $\mathbb{F}$, and $\mathrm{SL}_n(\mathbb{F})$ denote the subgroup of $\mathrm{GL}_n(\mathbb{F})$ consisting of matrices with determinant one. Let $I$ denote the identity matrix and $\mathrm{0}$ denote the zero matrix; the dimension of each is always clear from context. Let $\mathrm{S}_n$ denote the symmetric group and $P_\pi$ denote the permutation $\pi\in \mathrm{S}_n$ of the columns of $I$. For $X \subset \mathrm{M}_n(\mathbb{F})$, let $X^k \subset \mathrm{M}_n(\mathbb{F})$ denote the set of $k$-fold products of elements of $X$. Let $\mathrm{BL}_{m,n}(\mathbb{F})$ and $\mathrm{BU}_{m,n}(\mathbb{F})$ denote the subgroups of $\mathrm{SL}_{m+n}(\mathbb{F})$ consisting of block lower and upper unitriangular matrices, respectively, with block partition $\{1,...,m\}$ and $\{m+1,...,m+n\}$:
\begin{align*}
    \mathrm{BL}_{m,n}(\mathbb{F}) &= \bigg\{ \begin{bmatrix} I & 0 \\ A & I \end{bmatrix} \, \bigg| \, A \in \mathrm{M}_{n,m}(\mathbb{F}) \bigg\},\\
    \mathrm{BU}_{m,n}(\mathbb{F})&= \bigg\{ \begin{bmatrix} I & A \\ 0 & I \end{bmatrix} \, \bigg| \,A \in \mathrm{M}_{m,n}(\mathbb{F}) \bigg\}.
\end{align*}
In this work, we prove the existence of the following block unitriangular factorization of the special linear group.
\begin{theorem}\label{thm:sl}
For every $M \in \mathrm{SL}_{2n}(\mathbb{F})$, there exists $A_1,...,A_6 \in \mathrm{M}_n(\mathbb{F})$ such that
$$ M = \begin{bmatrix} I & 0 \\ A_1 & I \end{bmatrix} \begin{bmatrix} I & A_2 \\ 0 & I \end{bmatrix} \begin{bmatrix} I & 0 \\ A_3 & I \end{bmatrix} \begin{bmatrix} I & A_4 \\ 0 & I \end{bmatrix}\begin{bmatrix} I & 0 \\ A_5 & I \end{bmatrix} \begin{bmatrix} I & A_6 \\ 0 & I \end{bmatrix} .$$
For every $m+n>3$, there exists some $M \in \SL_{m+n}(\mathbb{F})$ such that $M \not \in \big[\mathrm{BL}_{m,n}(\mathbb{F}) \cup \mathrm{BU}_{m,n}(\mathbb{F})\big]^5$. Furthermore, if $\mathbb{F}$ has at least four elements, then the lower bound holds independent of indexing: there exists $M \in \SL_{m+n}(\mathbb{F})$ such that $P_{\pi} M P_{\pi^{-1}} \not \in \big[\mathrm{BL}_{m,n}(\mathbb{F}) \cup \mathrm{BU}_{m,n}(\mathbb{F})\big]^5$ for all permutations $\pi \in \mathrm{S}_{m+n}$.
\end{theorem}
We have an analogous theorem for the general linear group. Let $\mathrm{T}_{m,n}(\mathbb{F})$ denote the set of matrices of the form $\begin{bsmallmatrix} B & 0 \\ A & C \end{bsmallmatrix}$ or $\begin{bsmallmatrix} B & A \\ 0 &  C \end{bsmallmatrix}$, with $B \in \mathrm{GL}_m(\mathbb{F})$ and $C \in \mathrm{GL}_n(\mathbb{F})$ both diagonal. We prove the following result.

\begin{theorem} \label{thm:gl}
For every $M \in \mathrm{GL}_{2n}(\mathbb{F})$ and diagonal $D \in \mathrm{GL}_{n}(\mathbb{F})$ with $\mathrm{det}(D) = \mathrm{det}(M)$, there exists $A_1,...,A_6 \in \mathrm{M}_n(\mathbb{F})$ such that 
$$ M = \begin{bmatrix} I & 0 \\ A_1 & I \end{bmatrix} \begin{bmatrix} I & A_2 \\ 0 & I \end{bmatrix} \begin{bmatrix} I & 0 \\ A_3 & I \end{bmatrix} \begin{bmatrix} I & A_4 \\ 0 & I \end{bmatrix}\begin{bmatrix} I & 0 \\ A_5 & I \end{bmatrix} \begin{bmatrix} D & A_6 \\ 0 & I \end{bmatrix}.$$
Furthermore, for every $m+n>3$, there exists $M \in \mathrm{SL}_{m+n}(\mathbb{F})$ such that $M \not \in \big[\mathrm{T}_{m+n}(\mathbb{F})\big]^5$.
\end{theorem}

The factorizations of Theorems \ref{thm:sl} and \ref{thm:gl} are efficiently computable. Our construction is fairly unique among block matrix factorizations.\footnote{One somewhat similar construction is the $LU \tilde L$ factorization proposed by Serre and P{\"u}schel, where $L,\tilde L$ are block lower unitriangular and $U$ is block upper triangular; the authors studied how close (in a rank sense) such a factorization can be to block diagonal \cite{serre2016generalizing}.} It can be viewed as a generalized version of a block LU factorization. One benefit of our construction is that such a factorization always exists, whereas the existence of a block LU factorization relies on the invertibility of the upper left block. Moreover, the above theorems serve as general results regarding the representational power of alternating linear updates. 

In fact, Theorem \ref{thm:gl} immediately solves an open problem regarding affine coupling networks in machine learning. An affine coupling block is a function $f:\mathbb{R}^{m+n}\rightarrow \mathbb{R}^{m+n}$ of the form $f(x_u,x_v) = (x_u,x_v \odot \exp(s(x_u))+t(x_u)) $, where $x_u \in \mathbb{R}^m$, $x_v \in \mathbb{R}^{n}$ (typically, $m \approx n$), $s,t:\mathbb{R}^{m} \rightarrow \mathbb{R}^{n}$, and $\odot$ is the entry-wise product. A version of these functions (with $s = 0$) was originally introduced by Dinh, Krueger, and Y. Bengio in their NICE deep learning model \cite{dinh2014nice}. Dinh, Sohl-Dickstein, and S. Bengio expanded that work to real NVP (non-volume preserving) transformations (e.g., the general formulation above) \cite{dinh2016density}. These papers led, in part, to the popularization of normalizing flows in machine learning, a general class of diffeomorphisms that map some standard distribution (say, a standard Gaussian vector) to a more complex one. We refer the reader to \cite{kobyzev2020normalizing,papamakarios2021normalizing} for two excellent surveys of this emerging field. As discussed in \cite{papamakarios2021normalizing}, the ``question of foremost importance'' is the expressive power of such models, even when restricted to simple inputs. Recently, Koehler, Mehta, and Risteski studied the expressive power of linear affine couplings, i.e., matrices of the form $\begin{bsmallmatrix} I & 0 \\ A & D \end{bsmallmatrix}$ and $\begin{bsmallmatrix} D & A \\ 0 & I \end{bsmallmatrix}$ for $A \in \mathrm{M}_n(\mathbb{R})$ and diagonal $D \in \GL_n(\mathbb{R})$ with strictly positive entries \cite{koehler2021representational}. They posed the following question: how many linear affine coupling layers are needed to represent an arbitrary orientation-preserving matrix? They produced a $47$-layer construction and showed that at least five layers are necessary. Theorem \ref{thm:gl} answers their question exactly.

\begin{corollary}\label{cor:affine}
    Every matrix $M \in \GL_{2n}(\mathbb{R})$ with $\det(M)>0$ can be represented by a depth-six linear affine coupling network. In addition, for every $n>1$, there exists $M \in \SL_{2n}(\mathbb{R})$ for which six layers are necessary, i.e., $M$ cannot be exactly represented by a depth-five linear affine coupling network.
\end{corollary}

The lower bound of Corollary \ref{cor:affine} immediately implies one for the non-linear setting, by considering the function $M x$, and applying a Jacobian argument; see \cite[Corollary 6]{koehler2021representational} for details.
Theorem \ref{thm:sl} gives an analogous result for the aforementioned NICE model (where $D = I$), with an additional lower bound independent of indexing, e.g., a learned partition cannot do uniformly better than an arbitrary one.  The improvement in construction from depth $47$ to depth six leads to a significant practical difference in terms of architecture. For example, since permutation matrices can be represented with six layers, the choice of partition may be of limited importance. Furthermore, the improved construction has consequences for maximum likelihood estimation, as Corollary \ref{cor:affine} implies that the distributions representable as the application of a six-layer linear affine coupling network to $N(0,I)$ are exactly the set of $N(0,\Sigma)$ with $\Sigma$ invertible; see \cite[Appendix A.2]{koehler2021representational} for details.

\section*{Proof of Theorems \ref{thm:sl} and \ref{thm:gl}}\label{sec:sl}

We construct the factorizations of Theorems \ref{thm:sl} and \ref{thm:gl} first by showing that matrices with a non-singular upper-right block can be represented with five layers (Lemmas \ref{lm:sl_generic} \& \ref{lm:sl2gf2_generic}), and then by proving that it costs only one layer to convert any matrix into one with a non-singular upper-right block (Lemma \ref{lm:linear_algebra}). We prove matching lower bounds by analyzing the class of block diagonal (Lemma \ref{lm:lower}) and diagonal (Lemma \ref{lm:sl_perm}) matrices that can represented with five layers. 

In what follows, we make use of the theory of commutators, i.e., elements of a group $G$ of the form $[g,h]:=g^{-1}h^{-1}gh$ for some $g,h\in G$. We recall the following consequence of a combination of results of R.C. Thompson.

\begin{lemma}[\cite{thompson1960thesis,thompson1961commutators}]\label{lm:comm}
    If $\mathrm{SL}_n(\mathbb{F}) \ne \mathrm{SL}_2(\mathrm{GF}(2))$, then every element is a commutator of $\mathrm{GL}_n(\mathbb{F})$.
\end{lemma}

Furthermore, given $A \in \mathrm{SL}_n(\mathbb{F}) \ne \mathrm{SL}_2(\mathrm{GF}(2))$, $X,Y \in \mathrm{GL}_n(\mathbb{F})$ satisfying $A = [X,Y]$ are efficiently computable; see \cite{thompson1960thesis,thompson1961commutators} for details. Using Lemma \ref{lm:comm} and well-chosen block unitriangular matrices, we produce a five-layer decomposition for matrices with a non-singular upper right block.

\begin{lemma}\label{lm:sl_generic}
Let $M =  \begin{bmatrix} M_1 & M_2\\ M_3 & M_4 \end{bmatrix} \in \mathrm{SL}_{2n}(\mathbb{F}) \ne \mathrm{SL}_{4}(\mathrm{GF}(2))$ and $M_2 \in \mathrm{GL}_n(\mathbb{F})$. Then
$$ M = \begin{bmatrix} I & 0 \\ A_1 & I \end{bmatrix} \begin{bmatrix} I & A_2 \\ 0 & I \end{bmatrix} \begin{bmatrix} I & 0 \\ A_3 & I \end{bmatrix} \begin{bmatrix} I & A_4 \\ 0 & I \end{bmatrix}\begin{bmatrix} I & 0 \\ A_5 & I \end{bmatrix} ,$$
where 
\begin{align*}
A_1 &= M_4 M_2^{-1} + M_2^{-1}X^{-1} Y^{-1}(I-X)-M_2^{-1}X^{-1}, \\
A_2 &= X M_2, \\
A_3 &= M_2^{-1} X^{-1}(Y-I),\\
A_4 &= Y^{-1}(I-X)M_2, \\
A_5 &= M_2^{-1} (M_1-Y),
\end{align*}
and $X,Y \in \mathrm{GL}_n(\mathbb{F})$ satisfy $[X,Y]= M_2(M_4M_2^{-1} M_1 - M_3)$.
\end{lemma}

\begin{proof}
 $\det \big[M_2(M_4M_2^{-1} M_1 - M_3) \big] = \det M$ \cite[Sec. 0.8.5]{horn2012matrix}, and so, by Lemma \ref{lm:comm}, there exists $X,Y \in \mathrm{GL}_n(\mathbb{F})$ with $[X,Y]= M_2(M_4M_2^{-1} M_1 - M_3)$. The result follows from a short computation: 
 \begin{align*} \begin{bmatrix} I & 0 \\ A_1 & I \end{bmatrix} \begin{bmatrix} I & A_2 \\ 0 & I \end{bmatrix} \begin{bmatrix} I & 0 \\ A_3 & I \end{bmatrix} &=\begin{bmatrix}  Y &  XM_2 \\ {\scriptstyle M_4 M_2^{-1}Y  - M_2^{-1} [X,Y]  }& {\scriptstyle (M_4 M_2^{-1} X + M_2^{-1}[X,Y]Y^{-1}(I-X))M_2} \end{bmatrix} , \\
\begin{bmatrix} I & A_4 \\ 0 & I \end{bmatrix}\begin{bmatrix} I & 0 \\ A_5 & I \end{bmatrix} &=\begin{bmatrix}Y^{-1}(M_1 -X(M_1 - Y)) & Y^{-1}(I-X) M_2 \\ M_2^{-1}(M_1-Y) & I
    \end{bmatrix},
    \end{align*}
    and, given $[X,Y]= M_2(M_4M_2^{-1} M_1 - M_3)$, their product equals $M$.
\end{proof}

Though the proof of Lemma \ref{lm:sl_generic} is quite short, it gives little motivation for the choice of matrices $A_1,...,A_5$. This exact choice can only be justified by a detailed analysis of both the representational power of three layers and the image of any matrix under a two-layer transformation. Unfortunately, $\mathrm{SL}_{4}(\mathrm{GF}(2))$ cannot be treated using Lemma \ref{lm:sl_generic}, as the matrices $\begin{bsmallmatrix} 1 & 1 \\ 0 & 1 \end{bsmallmatrix}$, $\begin{bsmallmatrix} 1 & 0 \\ 1 & 1 \end{bsmallmatrix}$, $\begin{bsmallmatrix} 0 & 1 \\ 1 & 0 \end{bsmallmatrix}$ are not commutators of $\mathrm{SL}_{2}(\mathrm{GF}(2))$. Despite this, elements of $\mathrm{SL}_{4}(\mathrm{GF}(2))$ with non-singular upper right block can still be represented as the product of five block unitriangular matrices, which, given the small group size, is easily verified by exhaustive search.\footnote{See repository \cite{ourrepo} for a short computer-assisted proof (using the Julia programming language \cite{bezanson2017julia}); the program terminates in under a second on a personal computer. It is also possible to prove Lemma \ref{lm:sl2gf2_generic} via an involved case analysis. The details are left to the interested reader.}

\begin{lemma}[\cite{ourrepo}]\label{lm:sl2gf2_generic}
Let $M =  \begin{bsmallmatrix} M_1 & M_2\\ M_3 & M_4 \end{bsmallmatrix} \in \SL_{4}(\GF(2))$ and $M_2 \in \SL_2(\GF(2))$. Then there exists $A_1,...,A_5 \in \mathrm{M}_2(\GF(2))$ such that
 $M = \begin{bsmallmatrix} I & 0 \\ A_1 & I \end{bsmallmatrix} \begin{bsmallmatrix} I & A_2 \\ 0 & I \end{bsmallmatrix} \begin{bsmallmatrix} I & 0 \\ A_3 & I \end{bsmallmatrix} \begin{bsmallmatrix} I & A_4 \\ 0 & I \end{bsmallmatrix}\begin{bsmallmatrix} I & 0 \\ A_5 & I \end{bsmallmatrix}.$
\end{lemma}

The desired factorizations of Theorems \ref{thm:sl} and \ref{thm:gl} follow from the application of Lemmas \ref{lm:sl_generic} and \ref{lm:sl2gf2_generic} to the product $M \begin{bsmallmatrix} B & A \\ 0 & I\end{bsmallmatrix}$ for some diagonal $B \in \GL_n(\mathbb{F})$ satisfying $\det(B) = \det (M)^{-1}$ and $A \in \mathrm{M}_n(\mathbb{F})$ satisfying $M_1 A + M_2 \in \GL_n(\mathbb{F})$. That such a matrix $A$ exists is a consequence of the following simple lemma, as $M\in \GL_{2n}(\mathbb{F})$ implies $\mathrm{coker}(M_1) \cap \mathrm{coker}(M_2) = 0$.

\begin{lemma}\label{lm:linear_algebra}
For any $A,B \in  \mathrm{M}_n(\mathbb{F})$, there exists $C \in \mathrm{M}_n(\mathbb{F})$ such that $CA + B \in \mathrm{GL}_n(\mathbb{F})$ if and only if $\mathrm{ker}(A) \cap \mathrm{ker}(B)=0$.
\end{lemma}

\begin{proof}
    $\mathrm{ker}(A) \cap \mathrm{ker}(B)=0$ is clearly necessary, as $\mathrm{ker}(A) \cap \mathrm{ker}(B) \subset \mathrm{ker}(CA+B)$. The converse also follows quickly. Simply choose $C$ to be any matrix for which $\mathrm{im}(C)$ is a complement of $\mathrm{im}(B)$ and $\mathrm{ker}(C) \cap \{ A x \, | \, x \in \mathrm{ker}(B) \} = 0$. Such a matrix always exists, as $\mathrm{ker}(A) \cap \mathrm{ker}(B)=0$ and the rank-nullity theorem together imply $\dim(\{ A x \, | \, x \in \mathrm{ker}(B) \}) = \dim(\mathrm{im}(C)).$
\end{proof}

We now consider the lower bounds of Theorems \ref{thm:sl} and \ref{thm:gl}. We have the following lemma regarding the representation of block diagonal matrices.

\begin{lemma}\label{lm:lower}
If $\begin{bsmallmatrix} M_1 & 0 \\ 0 & M_4 \end{bsmallmatrix} \in \big[\mathrm{T}_{m,n}(\mathbb{F})\big]^5$, $M_1\in \GL_m(\mathbb{F})$, $M_4 \in \GL_n(\mathbb{F})$, then there exist diagonal matrices $D \in \GL_m(\mathbb{F})$ and $\widetilde D \in \GL_n(\mathbb{F})$ such that
$$m \cdot 1 + \mathrm{trace}(M_4 \widetilde D) = n \cdot 1 + \mathrm{trace}(M_1^{-1} D).$$
\end{lemma}

\begin{proof}
It suffices to consider an $LULUL$ factorization, as the lemma statement is invariant under transpose. Suppose $$\begin{bmatrix} M_1 & 0 \\ 0 & M_4 \end{bmatrix}  = \begin{bmatrix} B_1 & 0 \\ A_1 & C_1 \end{bmatrix} \begin{bmatrix} B_2 & A_2 \\ 0 & C_2 \end{bmatrix} \begin{bmatrix} B_3 & 0 \\ A_3 & C_3 \end{bmatrix} \begin{bmatrix} B_4 & A_4 \\ 0 & C_4 \end{bmatrix}\begin{bmatrix} B_5 & 0 \\ A_5 & C_5 \end{bmatrix}$$ for some $A_1, A_3, A_5 \in \mathrm{M}_{n,m}(\mathbb{F})$, $A_2,A_4 \in \mathrm{M}_{m,n}(\mathbb{F})$, and diagonal matrices $B_1, ..., B_5 \in \GL_{m}(\mathbb{F})$ and $C_1, ..., C_5 \in \mathrm{GL}_n(\mathbb{F})$. We have
$$\begin{bmatrix} B_1 & 0 \\ A_1 & C_1 \end{bmatrix} \begin{bmatrix} B_2 & A_2 \\ 0 & C_2 \end{bmatrix} \begin{bmatrix} B_3 & 0 \\ A_3 & C_3 \end{bmatrix} =\begin{bmatrix} B_1 B_2 B_3 + B_1 A_2 A_3 & B_1 A_2 C_3 \\ A_1B_2 B_3 + C_1 C_2 A_3 + A_1 A_2 A_3 & C_1 C_2 C_3 +A_1 A_2C_3\end{bmatrix}$$
and
\begin{align*}
\begin{bmatrix}
    M_1 & 0 \\ 0 & M_4
\end{bmatrix}\begin{bmatrix} B_5 & 0 \\ A_5 & C_5 \end{bmatrix}^{-1} \begin{bmatrix} B_4 & A_4 \\ 0 & C_4 \end{bmatrix}^{-1} & =\begin{bsmallmatrix}
    M_1^{} & 0 \\ 0 & M_4^{}
\end{bsmallmatrix} \begin{bsmallmatrix} B_5^{-1} & 0 \\ - C_5^{-1} A_5 B_5^{-1} & C_5^{-1} \end{bsmallmatrix} \begin{bsmallmatrix} B_4^{-1} & -B_4^{-1} A_4 C_4^{-1} \\ 0 & C_4^{-1} \end{bsmallmatrix} \\
&=\begin{bsmallmatrix}
    M_1B_5^{-1} B_4^{-1} & -M_1B_5^{-1} B_4^{-1} A_4 C_4^{-1} \\ -M_4 C_5^{-1} A_5 B_5^{-1} B_4^{-1} & M_4C_5^{-1}(I+A_5 B_5^{-1} B_4^{-1} A_4)C_4^{-1}
\end{bsmallmatrix}.
\end{align*}
By setting these matrices equal and inspecting the upper left and right blocks, we deduce that
 $A_2 A_3 = B_1^{-1} M_1 B_5^{-1} B_4^{-1} - B_2 B_3$ and $A_4 C_4^{-1} = -B_4 B_5 M_1^{-1} B_1 A_2 C_3$. Using the former equality applied to the lower left block, 
$$-M_4 C_5^{-1} A_5 B_5^{-1} B_4^{-1} = A_1 B_1^{-1}M_1 B_5^{-1} B_4^{-1} + C_1 C_2 A_3,$$
which, all together, implies (using the lower right block)
    \begin{align*}
    M_4 C_5^{-1}C_4^{-1} &=(C_1 C_2 + A_1 A_2) C_3 - M_4 C_5^{-1} A_5 B_5^{-1} B_4^{-1} A_4 C_4^{-1} \\
    &=(C_1 C_2 + A_1 A_2) C_3 + (A_1 B_1^{-1}M_1 B_5^{-1} B_4^{-1}  + C_1 C_2 A_3)(-B_4 B_5 M_1^{-1} B_1 A_2 C_3) \\
    &= C_1 C_2 C_3 -C_1 C_2 A_3 B_4 B_5 M_1^{-1} B_1 A_2 C_3.
\end{align*}
Therefore,
$$A_2(A_3 B_4 B_5 M_1^{-1}B_1) = I - B_2 B_3 B_4 B_5 M_1^{-1} B_1$$
and
$$ (A_3 B_4 B_5 M_1^{-1}B_1)A_2 = I - C_2^{-1} C_1^{-1} M_4 C_5^{-1} C_4^{-1} C_3^{-1}.$$
Taking the trace of each gives our desired result, as the product of two matrices has a fixed trace, independent of the order of operands.
\end{proof}

Consider the matrices $X \in \GL_m(\mathbb{F})$ and $Y \in \GL_n(\mathbb{F})$, $m,n>1$, defined as follows: \vspace{1 mm}
$$X(i,j) = \begin{cases} \; 1 & \text{if} \; j - i = 1 \mod m \\ \; 0 & \text{otherwise}\end{cases}\,, \qquad
    Y(i,j) = \begin{cases} 
     \; \delta(m\cdot 1 - n \cdot 1)& \text{if} \; i=j=1 \\
     \; (-1)^{m+n} & \text{if} \; i=n,j=1  \\
          \; 1 & \text{if} \; j - i = 1  \\   
     \; 0  &\text{otherwise}\end{cases}\,,\vspace{1 mm}$$
     where $\delta(\cdot)$ is the Kronecker delta function. We have $\det(X) = \det(Y) = (-1)^{m+1}$, $\mathrm{trace}(X D) = 0$ for all diagonal $D \in \GL_m(\mathbb{F})$, and, for every diagonal $\widetilde D \in \GL_n(\mathbb{F})$, $\mathrm{trace}(Y \widetilde D) \ne 0$ if and only if $m\cdot 1 = n \cdot 1$. Therefore, by Lemma \ref{lm:lower}, $\begin{bsmallmatrix} X^{-1} & 0 \\ 0 & Y \end{bsmallmatrix} \not\in \big[\mathrm{T}_{m,n}(\mathbb{F})\big]^5$. To complete our desired lower bound, we must briefly analyze the case when either $m$ or $n$ is equal to one. If, say, $n = 1$ and $m>2$, let us keep $X$ as above and set $Y = (-1)^{m+1}$, so that $\begin{bsmallmatrix} X^{-1} & 0 \\ 0 & Y \end{bsmallmatrix} \in \SL_{m+1}(\mathbb{F})$. By the analysis in the proof of Lemma \ref{lm:lower}, if $\begin{bsmallmatrix} X^{-1} & 0 \\ 0 & Y \end{bsmallmatrix} \in \big[\mathrm{T}_{m,n}(\mathbb{F})\big]^5$, then $I - \widehat D X D$ is a rank one matrix for some diagonal $\widehat D, D \in \GL_{m}(\mathbb{F})$. However, this is not possible, as $[I - \widehat D X D](1,1) =[I - \widehat D X D](2,2) =  1$ and $[I - \widehat D X D](2,1) = 0$. This completes the proof of Theorem \ref{thm:gl}.
When $\mathbb{F}$ has at least four elements, the lower bound for $\SL_n(\mathbb{F})$ holds independent of indexing. The following lemma completes the proof of Theorem \ref{thm:sl}.

\begin{lemma}\label{lm:sl_perm}
If $\mathbb{F}$ has at least four elements, then, for every $m+n>3$, there exists $M \in \mathrm{SL}_{m+n}(\mathbb{F})$ such that $P_{\pi} M P_{\pi^{-1}} \not \in \big[\mathrm{BL}_{m,n}(\mathbb{F}) \cup \mathrm{BU}_{m,n}(\mathbb{F})\big]^5$ for all permutations $\pi \in \mathrm{S}_{m+n}$.
\end{lemma}

\begin{proof}
Let $M$ be diagonal, with diagonal elements $g$, $h$, $(gh)^{-1}$ (not necessarily distinct), and $2n-3$ copies of $1$, for some $g,h \ne 1$ satisfying $gh \ne 1$. Such $g,h \in \mathbb{F}$ always exists when $\mathbb{F}$ has at least four elements (take any $g_1 ,g_2 \ne 0,1$ distinct; either $g_1^2 \ne 1$ or $g_1g_2 \ne 1$). Now suppose
$$M=\begin{bmatrix} M_1 & 0 \\ 0 & M_4 \end{bmatrix}  = \begin{bmatrix} I & 0 \\ A_1 & I \end{bmatrix} \begin{bmatrix} I & A_2 \\ 0 & I \end{bmatrix} \begin{bmatrix} I & 0 \\ A_3 & I \end{bmatrix} \begin{bmatrix} I & A_4 \\ 0 & I \end{bmatrix}\begin{bmatrix} I & 0 \\ A_5 & I \end{bmatrix}$$ for some $A_1, A_3, A_5 \in \mathrm{M}_{n,m}(\mathbb{F})$ and $A_2,A_4 \in \mathrm{M}_{m,n}(\mathbb{F})$. Repeating the same analysis as in the proof of Lemma \ref{lm:lower}, we find that 
$$A_2 (A_3 M_1^{-1})= I - M_1^{-1} \quad \text{ and } \quad (A_3 M_1^{-1} )A_2 = I - M_4.$$
The product of two matrices has a fixed set of non-zero characteristic roots, independent of the order of operands \cite[Theorem 1]{flanders1951elementary}. However, in total, exactly three elements of $I - M_1^{-1}$ and $I - M_4$ are non-zero. Therefore, there is no ordering and bipartition of the diagonal elements such that the non-zero characteristic roots, taken with multiplicity, of $I - M_1^{-1}$ and $I - M_4$ are the same, a contradiction. 
\end{proof}

\subsection*{Acknowledgements} The author thanks Louisa Thomas for improving the style of presentation.

{ 
	\bibliographystyle{plain}
	\bibliography{main} }

\end{document}